\documentclass[a4paper, 12pt]{article}

\usepackage[utf8]{inputenc}
\usepackage[english]{babel}
\usepackage{amsmath}
\usepackage{amssymb}
\usepackage{amsthm}
\usepackage{algorithm,algorithmicx,algpseudocode}
\usepackage{verbatim}
\usepackage{graphicx,color}
\usepackage{graphics}
\usepackage{subfigure}
\usepackage{mathrsfs}
\usepackage{epstopdf}

\newcommand{\T}{\mathcal{T}}
\newcommand{\E}{\mathcal{E}}

\newcommand{\V}{\mathcal{V}}

\newcommand{\fgref}[1]{Figure~\ref{#1}}

\newcommand{\defref}[1]{Definition~\ref{#1}}
\newcommand{\lemref}[1]{Lemma~\ref{#1}}

\newcommand{\thmref}[1]{Theorem~\ref{#1}}
\newcommand{\Enorm}[1]{||| #1 |||}
\newcommand{\Enormh}[1]{||| #1 |||_{\operatorname*{h}}}
\newcommand{\Enormd}[1]{||| #1 |||_{\operatorname*{h},\operatorname*{d}}}
\newcommand{\Enorma}[1]{||| #1 |||_{\operatorname*{h},\operatorname*{c}}}
\newcommand{\EnormaH}[1]{||| #1 |||_{\operatorname*{a},H}}
\newcommand{\EnormdH}[1]{||| #1 |||_{\operatorname*{d},H}}

\newcommand{\Enormhw}[2]{||| #1 |||_{h,#2}}
\newcommand{\EnormH}[1]{||| #1 |||_{H}}

\newcommand{\e}{\mathscr{E}}
\newcommand{\mscale}{A}

\newcommand{\VH}{\mathcal{V}_H}

\newcommand{\VmsH}{\mathcal{V}^{ms}_H}

\newcommand{\VmsLH}{\mathcal{V}^{ms,L}_{H}}

\newcommand{\Vh}{\mathcal{V}_h}
\newcommand{\Vfh}{\mathcal{V}^{\operatorname*{f}}} %_h
\newcommand{\Vfhw}[1]{\mathcal{V}^f(#1)}
\newcommand{\IcH}{\mathcal{I}^{\operatorname*{c}}_H}
\newcommand{\Lp}{\Pi_H}%{\Pi_1(\T_H)}

\newcommand{\ba}{C_{A}}
\newcommand{\Cbubble}{\ba}
\newcommand{\Ccontinuous}{C_{c}}
\newcommand{\Cstable}{C_{s}}
\newcommand{\Ccutoff}{C_{\zeta}}
\newcommand{\Cbas}{C_{\phi}}
\newcommand{\supp}{\operatorname*{supp}}
\newcommand{\enorm}[1]{\left|\left|\left|#1\right|\right|\right|}

\newtheorem{theorem}{Theorem}
\newtheorem{lemma}[theorem]{Lemma}
\newtheorem{definition}[theorem]{Definition}

\newtheorem{remark}[theorem]{Remark}

\usepackage[utf8]{inputenc}

\begin{document} 
\title{A discontinuous Galerkin multiscale method for convection-diffusion problems}

\author{Daniel Elfverson}

%\institute{Information Technology, Uppsala University, Box 337, SE-751 05, Uppsala, Sweden. Supported by The Göran Gustafsson Foundation and The Swedish Research Council.}

\date{}

\maketitle

\begin{abstract}
  We propose an discontinuous Galerkin local orthogonal decomposition
  multiscale method for convection-diffusion problems with rough,
  heterogeneous, and highly varying coefficients. The properties of
  the multiscale method and the discontinuous Galerkin method allows
  us to better cope with multiscale features as well as
  interior/boundary layers in the solution. In the proposed method the
  trail and test spaces are spanned by a corrected basis computed on
  localized patches of size $\mathcal{O}(H\log(H^{-1}))$, where $H$ is
  the mesh size. We prove convergence rates independent of the
  variation in the coefficients and present numerical experiments
  which verify the analytical findings.  %\keywords{multiscale method,
  %  discontinuous Galerkin, convergence, convection} \subclass{65N12,
  %  65N30}
\end{abstract}

\section{Introduction}
In this paper we consider numerical approximation of 
convection-diffusion problems with possible strong convection and with
rough, heterogeneous, and highly varying coefficients, without
assumption on scale separation or periodicity. This class of problems,
normally refereed to as multiscale problem, are know to be very
computational demanding and arise in many different areas of the
engineering sciences, e.g., in porous media flow and composite
materials.  More precisely, we consider the following
convection-diffusion equation: given any $f\in L^2(\Omega)$ we seek
$u\in H^1_0(\Omega)=\{v\in H^1(\Omega)\mid v|_{\Gamma}=0\}$ such that
\begin{equation}\label{eq:model}
  \begin{aligned}
    - \nabla \cdot A\nabla u + \mathbf{b}\cdot\nabla u &= f\quad\text{in }\Omega ,\\
  \end{aligned}
\end{equation}
is fulfilled in a weak sense, where $\Omega$ is the computational
domain with boundary $\Gamma$. The multiscale coefficients
$A,\mathbf{b}$ will be specified later.  There are two key issues
which make classical conforming finite element methods perform badly
for these kind of problems,
\begin{itemize}
\item the multiscale features of the coefficient need to be resolved
  by the finite element mesh and
\item strong convection leads to boundary and interior layers in the
  solution which need to be resolved.
\end{itemize}

To overcome the lack of performance using classical finite element
methods in the case of multiscale features in the coefficient many
different so called multiscale methods have been proposed, see
\cite{H95,HFMQ98,HW97,CEGH08,EHW00,EH09,EE03a,EE03b} among others,
which perform localized fine scale computations to construct a
different basis or a modified coarse scale operator. Common to the
aforementioned approaches is that the performance of the method rely
strongly on scale separation or periodicity of the diffusion
coefficients. There is also approaches which perform well without
scale separation or periodicity in the diffusion coefficient but to
high computational cost by either having to solve eigenvalue problems
\cite{BL11} or where the support of the localized patches is large
\cite{OZ11,BO10}. See also \cite{OZB14}.

In the variational multiscale method (VMS) framework \cite{H95,HFMQ98}
the solution space is split into coarse and fine scale
contribution. This idea was employed for multiscale problems in a
adaptive setting for classical finite element in \cite{LM07,M10,LM09b}
and to the discontinuous Galerkin (DG) method in \cite{EGM13}. A
further development is the local orthogonal decomposition (LOD) method,
see \cite{MP14ver1,HM14,HP13} for classical finite element and
\cite{EGMP13} for DG methods. The LOD operates in linear complexity
without any assumptions on scale separation or periodicity and the
trail and test spaces are spanned by a corrected basis function
computed on patches of size $\mathcal{O}(H\log(H^{-1}))$. The
LOD has e.g. been applied to eigenvalue problems \cite{MP14ver2},
non-linear elliptic problems \cite{HMP14}, non-linear Schrödinger
equation \cite{HMP13v2}, and in Petrov-Galerkin formulation
\cite{EGH14}.

There is a vast literature on numerical methods for convection
dominated problems, we reefer to \cite{JN81,HMM86,HFH89}, among
others. There has also been a lot of work on DG methods, we refer to
\cite{RH73,LR74,Bak77,JP86} for some early work and to
\cite{CKS00,Hes08,Riv08,PE12} and references therein for recent
development and a literature review.  DG methods exhibit attractive
properties for convection dominated problems, e.g., they have enhanced
stability properties, good conservation property of the state
variable, and the use of complex and/or irregular meshes are
admissible. For multiscale methods for convection-diffusion problems,
see e.g. \cite{AH14,Sod11,HO10}.

In this paper we extended the analysis of the discontinuous Galerkin
local orthogonal decomposition (DG-LOD) \cite{EGMP13} to
convection-diffusion problems. For problems with strong convection
using the standard LOD won't suffice, since convergence can no longer
be guarantied. Instead we propose to include the convective term in
the computations of the corrected basis functions. We prove
convergence results under some assumptions of the magnitude of the
convection and present a series of numerical experiment to verify the
analytic findings. For problems with weak convection it is not
necessary to include the convective part \cite{HMP14}.

The outline of this paper is as follows. In section \ref{sec:prel} the
discrete setting and underlying DG method is presented. In section
\ref{sec:MM} the multiscale decomposition, the DG-LOD, and the
corresponding convergence result are stated. In Section
\ref{sec:numerics} numerical experiments are presented. Finally, the
proofs for some of the theoretical results are given in Section
\ref{sec:proofs}.

\section{Preliminaries}\label{sec:prel}
In this section we present some notations and properties frequently
used in the paper. 

\subsection{Setting}
Let $\Omega\subset\mathbb{R}^d$ for $d=2,3$ be a polygonal domain with Lipschitz
boundary $\Gamma$. We assume that: the diffusion coefficients, $A\in
L^\infty(\Omega,\mathbb{R}^{d\times d}_{sym})$, has uniform spectral
bounds $0<\alpha,\beta<\infty$, defined by
\begin{equation}\label{eq:ba}
  0<\alpha:=\underset{x\in\Omega}{\operatorname{ess}\inf}\hspace{-1ex}
  \inf\limits_{v\in\mathbb{R}^{\operatorname*{d}}\setminus\{0\}}\hspace{-2ex}\dfrac{(A( x)v)\cdot v}{v\cdot v}\leq\underset{x\in\Omega}{\operatorname{ess}\sup}\hspace{-1ex}
  \sup\limits_{v\in\mathbb{R}^{\operatorname*{d}}\setminus\{0\}}\hspace{-2ex}\dfrac{(A( x)v)\cdot v
  }{v\cdot v}=:\beta<\infty,
\end{equation}
and the convective coefficient, $\mathbf{b}\in [W_\infty^1(\Omega)]^d$, is
divergence free
\begin{equation}\label{eq:c}
 \nabla\cdot\mathbf{b}(x) = 0  \text{ a.e. }x\in\Omega.
\end{equation}
We denote $\ba=(\beta/\alpha)^{1/2}$.

We will consider a coarse and a fine mesh, with mesh function $h$ and
$H$ respectively. Furthermore, we assume that the fine mesh resolve
and that the coarse mesh do not resolve the fine scale features in the
coefficients. Let $\T_{k}$, for $k=\{h,H\}$, denote a shape-regular
subdivision of $\Omega$ into (closed) regular simplexes or into
quadrilaterals/hexahedras ($d=2/d=3$), given a mesh function
$k:\T_{k}\to\mathbb{R}$ defined as $k:=\text{diam}(T)\in P_0(\T_{k})$
for all $T\in\T_{k}$. Also, let $\nabla_{k} v$ denote the
$\T_{k}$-broken gradient defined as $(\nabla v)\vert_T = \nabla
v\vert_T$ for all $T\in\T_{k}$. For simplicity we will also assume
that $\T_{k}$ is conforming in the sense that no hanging nodes are
allowed, but the analysis can easily be extend to non-conforming
meshes with a finite number of hanging nodes on each edge. Let $\hat
T$ be the reference simplex or (hyper)cube. We define
$\mathcal{P}_p(\hat T)$ to be the space of polynomials of degree less
than or equal to $p$ if $\hat T$ is a simplex, or the space of
polynomials of degree less than or equal to $p$, in each variable, if
$\hat T$ is a (hyper)cube. The space of discontinuous piecewise
polynomial function is defined by
\begin{equation}
  P_p(\T_{k}):=\{v:\Omega\to\mathbb{R}\mid \forall T\in\T_{k},\, v\vert_T\circ F_{T}\in\mathcal{P}_p(\hat{T})\},
\end{equation}
where $F_{T}:\hat T\to T$, $T\in\T_{k}$ is a family of element
maps. We will work with the spaces $\V_k:=P_1(\T_{k})$. Let
$\Pi_p(\T_{k}):L^2(\Omega)\to P_p(\T_{k})$ denote the $L^2$-projection
onto $P_p(\T_{k})$. Also, let $\E_{k}$ denote the set of all edges in
$\T_{k}$ where $\E_{k}(\Omega)$ and $\E_{k}(\Gamma)$ denote the
set of interior and boundary edges, respectively. Given that $T^+$ and
$T^-$ are two adjacent elements in $\T_{k}$ sharing an edge $e =
T^+\cap T^-\in\E_{k}(\Omega)$, let $\nu_e$ be the
the unit normal vector pointing from $T^-$ to $T^+$, and for
$e\in\E_{k}(\Gamma)$ let $\nu_e$ be outward unit normal of
$\Omega$. For any $v\in P_p(\T_{k})$ we denote the value on edge
$e\in\E(\Omega)$ as $v^\pm = v\vert_{e\cap T^\pm}$. The jump and
average of $v\in P_p(\T_{k})$ is defined as, $[v]=v^--v^+$ and
$\{v\}=(v^-+v^+)/2$ respectively for $e\in\E_{k}(\Omega)$, and $[v] =
\{v\} = v\vert_e$ for $e\in\E_{k}(\Gamma)$. For a real number $x$ we
define its negative part as $x^\ominus=1/2(|x|-x)$.

Let $0\leq C < \infty$ denote any generic constant that neither
depends on the mesh size or the variables $A$ and $\mathbf{b}$; then
$a \lesssim b$ abbreviates the inequality $a \leq C b$.

\subsection{Discontinuous Galerkin discretization}
For simplicity let the bilinear form
$a_{h}(\cdot,\cdot):\V_{h}\times\V_{h}\to\mathbb{R}$, given any mesh
function $h:\Omega\to P_0(\T_{h})$, be split into two parts
\begin{equation}
  a_{h}(u,v) := a^{\operatorname*{d}}_{h}(u,v) + a^{\operatorname*{c}}_{h}(u,v),
\end{equation}
where $a^{\operatorname*{d}}_{h}(\cdot,\cdot)$ represents the
diffusion part and $a^{\operatorname*{c}}_{h}(\cdot,\cdot)$
represents the convection part.  The diffusion part is
approximated using a symmetric interior penalty method
\begin{equation}
  \begin{aligned}
    a_{h}^{\operatorname*{d}}(u,v) &:= (\mscale\nabla_h u,\nabla_h v)_{L^2(\Omega)} + \sum_{e\in\E_h}\Big( \frac{\sigma_e}{h_e}([u],[v])_{L^2(e)} \\
    &\qquad  - (\{\nu_e\cdot\mscale\nabla u\},[v])_{L^2(e)} - (\{\nu_e\cdot\mscale\nabla v\},[u]_{L^2(e)})\Big),
  \end{aligned}
\end{equation}
where $\sigma_e$ is a constant, depending on the diffusion, large
enough to make $a_{h}^{\operatorname*{d}}(\cdot,\cdot)$ coercive. The
convective part is approximated by
\begin{equation}
  \begin{aligned}
    &    a_{h}^{\operatorname*{c}}(u,v) := (\mathbf{b}\cdot\nabla_h u,v)_{L^2(\Omega)}+\sum_{e\in\E_h(\Omega)}(b_e[u],[v])_{L^2(e)}\\ & \qquad-\sum_{e\in\E_h(\Omega)}(\nu_e\cdot\mathbf{b}[u],\{ v\})_{L^2(e)}+\sum_{e\in\E_h(\Gamma)}((\nu_e\cdot\mathbf{b})^{\ominus}u,v)_{L^2(e)},
  \end{aligned}
\end{equation}
where upwind is imposed choosing the stabilization term as $b_e =
|\mathbf{b}\cdot\nu_e|/2$ \cite{BMS04}. 

The following definitions and results are needed both on the fine and
coarse scale, for this sake let $k=\{h,H\}$. The energy norm on
$\V_{k}$ is given by
\begin{equation}
  \begin{aligned}
    \Enorm{v}^2_{k,\operatorname*{d}} &= \|A^{1/2}\nabla_{k} v \|^2_{L^2(\Omega)} + \sum_{e\in\E_{k}}\frac{\sigma_e}{k}\|[v]\|^2_{L^2(e)}, \\
    \Enorm{v}^2_{k,\operatorname*{c}} &=  \sum_{e\in\E_{k}}\|b_e^{1/2}[v]\|^2_{L^2(e)}, \\
    \Enorm{v}^2_{k}   &= \Enorm{v}^2_{k,\operatorname*{d}}+\Enorm{v}^2_{k,\operatorname*{c}}.
  \end{aligned}
\end{equation}
From Theorem 2.2 in \cite{KP03} we have that for each $v\in\V_{k}$,
there exist an averaging operator
$\mathcal{I}^c_{k}:\V_{k}\to\V_{k}\cap H^1(\Omega)$ with the following
property
\begin{equation}\label{eq:Ich}
  \|\nabla_{k}(v -\mathcal{I}^c_{k}v )\|_{L^2(\Omega)}+\|k^{-1}(v -\mathcal{I}^c_{k}v )\|_{L^2(\Omega)} \lesssim \sum_{e\in\E_{k}}\frac{1}{k}\|[v]\|^2_{L^2(e)}.
\end{equation}
In the error analysis we will also need a localized energy norm,
defined in a domain $\omega\subset\Omega$ (aligned with the mesh
$\T_k$) as
\begin{equation}
  \begin{aligned}
    \Enorm{v}^2_{k,\operatorname*{d},\omega} &= \|A^{1/2}\nabla_{k} v \|^2_{L^2(\omega)} + \sum_{\substack{e\in\E_{k}\\e\cap \bar\omega \neq 0}}\frac{\sigma_e}{k}\|[v]\|^2_{L^2(e)}, \\
    \Enorm{v}^2_{k,\operatorname*{c},\omega} &=  \sum_{\substack{e\in\E_{k}\\e\cap \bar\omega \neq 0}}\|b_e^{1/2}[v]\|^2_{L^2(e)}, \\
    \Enorm{v}^2_{k,\omega}   &= \Enorm{v}^2_{k,\operatorname*{d},\omega}+\Enorm{v}^2_{k,\operatorname*{c},\omega}.
  \end{aligned}
\end{equation}

\section{Multiscale method}\label{sec:MM}
In this section we preset the multiscale decomposition and extend the
results in \cite{EGMP13} to convection-diffusion problems. For the
constants in the convergence results to be stable we assume the
following relation of the convective term
\begin{equation}\label{eq:size_convection}
  \mathcal{O}\left(\frac{\| H\mathbf{b}\|_{L^\infty(\Omega)}}{\alpha}\right) \leq 1
\end{equation}
How the magnitude of \eqref{eq:size_convection} affects the
convergence of the method is investigated in the numerical
experiments.

\subsection{Multiscale decomposition}\label{sec:MD}
In order to do the multiscale decomposition the problem is divided
into a coarse and a fine scale. To this end let $\T_H$ and $\T_h$,
with the respective mesh function $H$ and $h$, denote the two
different subdivisions, where $\T_h$ is constructed by some (possible
adaptive) refinements of $\T_H$.

The aim of this section is to construct a coarse finite element space
based on $\T_H$, which takes the fine scale behavior of the data into
account. We assume that the mesh $\T_h$ resolves the variation in the
data, i.e., the solution to: find $u_h\in\Vh$ such that
\begin{equation}\label{eq:dg}
  a_h(u_h,v) = F(v)\quad\text{for all }v\in\Vh,
\end{equation}
gives a sufficiently good approximation of the weak solution $u$ to
\eqref{eq:model}. Note however that $u_h$ never have to be computed in
practice, it only acts as a reference solution.  We introduce a coarse
projection operator $\Lp:=\Pi_1(\T_H)$ and let the fine scale reminder
space be defined by the kernel of $\Lp$, i.e.,
\begin{equation}
  \Vfh := \{v\in\Vh \mid \Lp v = 0 \}\subset\Vh.
\end{equation}
The coarse projection operator has the following approximation and
stability properties.
\begin{lemma}\label{lem:stable}
  For any $v\in\Vh$ and $T\in\T_H$, the approximation property
  \begin{equation}\label{eq:Lp_approx}
    H\vert_T^{-1}\|v-\Lp v\|_{L^2(T)} \lesssim \alpha^{-1/2}\Enormhw{v}{T},
  \end{equation}
  and stability estimate
  \begin{equation}
    \EnormH{\Lp v} \lesssim \Cstable\Enormh{v},
  \end{equation}
  is satisfied, with
  \begin{equation}\label{eq:K}
    \Cstable = \left(\ba^2 + \frac{\|H\mathbf{b}\|_{L^\infty(\Omega)}}{\alpha}\right)^{1/2}.
  \end{equation}  
\end{lemma}
\begin{proof}
  The approximation property follows directly from
  \cite[Lemma~5]{EGMP13}. Let $\mathcal{C}_H:H^1\to H^1\cap \V_H$ be a
  Clément type interpolation operator proposed in
  \cite[Section~6]{CV99} which satisfy
  \begin{equation}
    \|\nabla \mathcal{C}_H u \|_{L^2(T)} + \|H^{-1}( u - \mathcal{C}_H u )\|_{L^2(T)} \lesssim \|\nabla u \|_{L^2(\omega_T^1)},
  \end{equation}
  where $\omega_T^1=\text{int}(\cup \{T'\in T_H\mid T \cap T \neq
  0\})$ are the union of all elements that share a edge with $T$. We
  define the conforming function $v_c=\mathcal{C}_H\mathcal{I}^c_hv$
  using averaging operator in \eqref{eq:Ich}.  We obtain
  \begin{equation}\label{eq:Lp_stability}
    \begin{aligned}
    &\EnormH{\Lp v}^2 = \sum_{T\in\T_H}\|A^{1/2}\nabla(\Lp v- \Pi_0v)\|^2_{L^2(T)}
\\
    &\qquad+\hspace{-0.5cm}\sum_{e\in \E_h(\Omega\cup\Gamma_D)}\hspace{-0.2cm}\Big(\frac{\sigma}{H}\|[v_c - \Lp v]\|^2_{L^2(e)} + \|b_e^{1/2}[v_c-\Lp v]\|^2\Big)\\
    &\lesssim \sum_{T\in\T_H}\beta\left(\frac{1}{H^2}\| v- \Pi_0v\|^2_{L^2(T)}+\left(\frac{1}{H^2}+\frac{\|\mathbf{b}\|_{L^\infty(T)}}{H}\right)\|v_c - v\|^2_{L^2(T)}\right)\\
    \end{aligned}
  \end{equation}
  using that $\Pi_0:=\Pi_0(\T_H)$ is the $L^2$-projection onto
  constants, a trace inequality, and stability of $\Lp$. Next, using that
  \begin{equation}
    \begin{aligned}
      \|\mathcal{C}_H\mathcal{I}^c_hv-v\|_{L^2(\Omega)} &\leq \|\mathcal{C}_H\mathcal{I}^c_hv-\mathcal{I}^c_hv\|_{L^2(\Omega)}+\|\mathcal{I}^c_hv-v\|_{L^2(\Omega)} \\
      & \lesssim H \|\nabla \mathcal{I}^c_hv\|_{L^2(\Omega)} + \|\mathcal{I}^c_hv-v\|_{L^2(\Omega)} \\
      & \lesssim \alpha^{-1/2}H\Enormh{v}
    \end{aligned}
  \end{equation}
  in \eqref{eq:Lp_stability} concludes the proof.
\end{proof}

The following lemma shows that for every $v_H\in\VH$ there exist
a (non-unique) $v\in\Lp^{-1}v_H\in\Vh$ in the preimage of $\Lp$
which is $H^1(\Omega)$ conforming.
\begin{lemma}\label{lem:bubble}
  For each $v_H\in\V_H$, there exist a $v\in\V_h\cap H^1(\Omega)$ such
  that $\Lp v = v_H$, $\Enormh{v}\lesssim\Cbubble\EnormH{v_H}$, and
  $\supp(v)\subset\supp(\IcH v_H)$.
\end{lemma}
\begin{proof}
  Follows directly from \cite[Lemma~6]{EGMP13}, since $v\in
  H^1(\Omega)$.
\end{proof}

The next step is to split any $v\in\Vh$ into some coarse part based on
$\T_H$, such that the fine scale reminder in the space $\Vfh$ is
sufficiently small. A naive way to do this splitting is to use a
$L^2$-orthogonal split. An alternative definition of the coarse space
is $\VH=\Lp \Vh$. A set of basis functions that span $\VH$ is the
element-wise Lagrange basis functions $\{\lambda_{T,j}\mid T\in\T_H,\,
j=1,\dots,r\}$ where $r=(1+d)$ for simplexes or $r = 2^d$ for
quadrilaterals/hexahedra. The space $\VH$ is known to give poor
approximation properties if $\T_H$ does not resolve the variable
coefficients in \eqref{eq:model}. We will use another choice, see
\cite{MP14ver1,EGMP13}, based on $a_h(\cdot,\cdot)$, to construct a
space of corrected basis functions. To this end, we define a fine
scale projection operator $\mathfrak{F}:\V_h\to\Vfh$ by
\begin{equation}\label{eq:fineProjection}
  a_h(\mathfrak{F}v,w) = a_h(v,w)\quad\text{for all }w\in\Vfh,
\end{equation}
and let the corrected coarse space be defined as
\begin{equation}\label{eq:fineproj}
  \VmsH :=(1-\mathfrak{F})\VH.
\end{equation}
The corrected space are spanned by corrected basis functions $\VmsH
:=\{\lambda_{T,j} - \phi_{T,j}\mid T\in\T_H,\, j=1,\dots,r\}$ which
can be computed as: for all $T\in\T_H,\, j=1,\dots,r$ find
$\phi_{T,j}\in\Vfh$ such that
\begin{equation}\label{eq:idealBasis}
  a_h(\phi_{T,j},v) =  a_h(\lambda_{T,j},v)\quad\text{for all }v\in\Vfh.
\end{equation}
Note that,
$\text{dim}(\VmsH)=\text{dim}(\VH)$. From \eqref{eq:fineproj} we have
that any $v_h\in\V_h$ can be decomposed into a coarse
$v^{ms}_H\in\VmsH$ and a fine $v^f\in\Vfh$ scale contribution, $v_h =
v^{ms}_H+v^f$.
\begin{lemma}[Stability of the corrected basis
  function]\label{lem:bas} For all $T\in\T_H,\, j=1,\dots,r$, the
  following estimate
\begin{equation}
  \Enormh{\phi_{T,h}-\lambda_{T,j}} \lesssim \Cbas \beta^{1/2}\|H^{-1}\lambda_{T,j}\|_{L^2(\Omega)}
\end{equation}
holds, where $\Cbas = (\ba^2 +
\|H\mathbf{b}\|_{L^\infty(\Omega)}\alpha^{-1})^{1/2}$.
\end{lemma}
\begin{proof}
Let $v = \lambda_{T,j}-b_{T,j}\in\Vfh$, where $b_{T,j}\in H^1_0(T)$, $\Lp b_{T,j} = \lambda_{T,j}$, $\Enormh{b_{T,j}}\leq\ba\EnormH{\lambda_{T,j}}$ from \lemref{lem:bubble}. We have
\begin{equation}\label{eq:stabelBas1}
  \begin{aligned}
 &   \Enormh{\phi_{T,h}-\lambda_{T,j}}^2 \lesssim  a_h(\phi_{T,h}-\lambda_{T,j},\phi_{T,h}-\lambda_{T,j}) \\
    &    =  a_h(\phi_{T,h}-\lambda_{T,j},v-\lambda_{T,j})  =  a_h(\phi_{T,h}-\lambda_{T,j},b_{T,j}) \\
    &    =  a_h^d(\phi_{T,h}-\lambda_{T,j},b_{T,j})+\left(\mathbf{b}\cdot\nabla_h(\phi_{T,h}-\lambda_{T,j}),b_{T,j}\right)_{L^2(\Omega)}.
  \end{aligned}
\end{equation}
Using that the diffusion part in \eqref{eq:stabelBas1} of the bilinear
form is continuous in $(\Vh\times\Vh)$ with the constant $\ba$,
\lemref{lem:bubble}, and a inverse inequality, we get
\begin{equation}
  \begin{aligned}
  a_h^d(\phi_{T,h}-\lambda_{T,j},b_{T,j}) &\lesssim \ba\Enormh{\phi_{T,h}-\lambda_{T,j}}\Enormh{b_{T,j}} \\
& \lesssim \ba^2 \Enormh{\phi_{T,h}-\lambda_{T,j}}\EnormH{\lambda_{T,j}}\\
& \lesssim \ba^2\beta^{1/2} \Enormh{\phi_{T,h}-\lambda_{T,j}}\|H^{-1}\lambda_{T,j}\|_{L^2(T)}.
  \end{aligned}
\end{equation}
For the convection part in \eqref{eq:stabelBas1}, we have
\begin{equation}
  \begin{aligned}
&  (\mathbf{b}\cdot\nabla_h(\phi_{T,h}-\lambda_{T,j}),b_{T,j})_{L^2(\Omega)} \\ 
& \lesssim \|H\mathbf{b}\cdot\nabla_h(\phi_{T,h}-\lambda_{T,j})\|_{L^2(\Omega)}\|H^{-1}b_{T,j}\|_{L^2(\Omega)} \\
&\lesssim   \|H\mathbf{b}\|_{L^\infty(\Omega)}\|\nabla_h(\phi_{T,h}-\lambda_{T,j})\|_{L^2(\Omega)} \|H^{-1}\lambda_{T,j}\|_{L^2(\Omega)},
  \end{aligned}
\end{equation}
and obtain
\begin{equation}
  \Enormh{\phi_{T,h}-\lambda_{T,j}} \leq \Cbas\beta^{1/2}\|H^{-1}\lambda_{T,j}\|_{L^2(\Omega)}.
\end{equation}
with $\Cbas = (\ba^2 + \|H\mathbf{b}\|_{L^\infty(\Omega)}\alpha^{-1})^{1/2}$.
\end{proof}

\subsection{Ideal discontinuous Galerkin multiscale method }
\label{sec:IdealDGMM}
An ideal multiscale method seeks $u^{ms}_H\in\VmsH$ such that
\begin{equation}\label{eq:ideal}
  a_h(u^{ms}_H,v) = F(v)\quad\text{for all }v\in\VmsH.
\end{equation}
Note that, to construct in the space $\VmsH$ a variational problem
has to be solved on the whole domain $\Omega$ for each basis function,
which is not feasible for real computations. The following theorem
shows the convergence of the ideal (non-realistic) multiscale method.
\begin{theorem}\label{thm:ideal}
  Let $u_h\in\Vh$ be the solution to \eqref{eq:dg}, and
  $u^{ms}_H\in\VmsH$ be the solution to \eqref{eq:ideal}, then
  \begin{equation}
    ||| u_h - u^{ms}_H ||| \lesssim C_1 \alpha^{-1/2}||H(f-\Lp f)||_{L^2(\Omega)}
  \end{equation}
  holds, with $C_1 = \ba + \|H\mathbf{b}\|_{L^\infty(\Omega)}\alpha^{-1} $
\end{theorem}
\begin{proof} See Section \ref{sec:proofs}.
\end{proof}

\subsection{Discontinuous Galerkin multiscale method}
\label{sec:DGMM}
The fast decay of the corrected basis functions (\lemref{lem:decay}),
motivates us to solve the corrector functions on localized
patches. This introduces a localization error, but choosing the patch
size as $\mathcal{O}(H\log(H^{-1}))$ (\thmref{thm:convergence}) the
localization error has the same convergence rate as the ideal
multiscale method in \thmref{thm:ideal}. The corrector functions are
solved on element patches, defined as follows.
\begin{definition}\label{def:patches}
  For all $T\in\T_H$, let $\omega^L_T$ be a patch centered around
  element $T$ with size $L$, defined as
  \begin{equation}
    \begin{aligned}
      \omega^0_T &:=\mathrm{int}(T), \\
      \omega^L_T &:= \mathrm{int}(\cup\{T'\in T_H \mid T\cap \bar\omega^{L-1}_T\neq 0\}),\quad L=1,2,\dots.
    \end{aligned}
  \end{equation}
  See \fgref{fig:patch} for an illustration.
\end{definition}
\begin{figure}[t]
  \centering
  \includegraphics[width=0.4\textwidth]{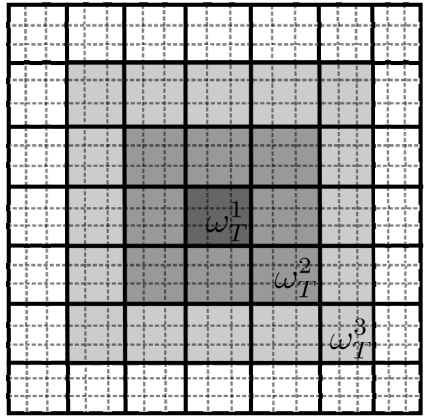}
  \caption{Example of a patch with one layer, $\omega^1_T$, two layers $\omega^2_T$, and three layers $\omega^3_T$, centered around element T.}
  \label{fig:patch}
\end{figure}

The localized corrector functions are calculated as follows: for all $\{T\in\T_H, j=1,\dots,r\}$ find $\phi^L_{T,j}\in\Vfhw{\omega^L_T} = \{v\in\Vfh \mid v|_{\Omega\setminus\omega^L_T} = 0\}$  such that
\begin{equation}\label{eq:locaBasis}
  a_h(\phi^L_{T,j},v) = a_h(\lambda_{T,j},v),\quad\text{for all }v\in\Vfhw{\omega^L_T}.
\end{equation}
The decay of the corrected basis function is given in the following lemma.
\begin{lemma}\label{lem:decay}
  For all $T\in\T_H$, $j=1,\dots,r$ where $\phi_{T,j}$ is the solution
  to \eqref{eq:idealBasis} and $\phi^{L}_{T,j}$ is the solution to
  \eqref{eq:locaBasis}, the following estimate
  \begin{equation}
    \enorm{\phi_{T,j}-\phi^{L}_{T,j}}_h \lesssim C_2\gamma^L \Enormh{\lambda_{T,j}-\phi^L_{T,j}}
  \end{equation}
  holds, where $L=\ell k$ is the size of the patch, $0<\gamma =
  (\ell^{-1}C_3)^{\frac{\ell(k-1)-1}{2\ell k(\ell+1)}}<1$, $C_2 =
  \Ccontinuous\Ccutoff(1+\Cbubble\Cstable)$, and  $C_3 =
  C(\ba^2+\|H\mathbf{b}\|_{L^\infty(\Omega)}\alpha^{-1})$, where $C$ is
  a generic constants neither depending on the mesh size, the size of
  the patches, or the problem data.
\end{lemma}
\begin{proof}
  See Section \ref{sec:proofs}.
\end{proof}

The space of localized corrected basis function is defined by $\VmsLH:=\{\phi^L_{T,j}-\lambda_{T,j}\mid T\in\T_H,\, r=1,\dots, r\}$. The DG multiscale method reads: find $u^{ms,L}_H\in\VmsLH$ such that
\begin{equation}\label{eq:localMM}
  a_h(u^{ms,L}_H,v) = F(v)\quad\text{for all }v\in\VmsLH.
\end{equation}
An error bound for the DG multiscale method  using a localized corrected basis is given in \thmref{thm:convergence}. Note that it is only the first term $\Enormh{u-u_h}$ in \thmref{thm:convergence} that depends on the regularity of $u$.
\begin{theorem}\label{thm:convergence}
  Let $u_h\in\Vh$ and $u^{ms,L}_H\in\VmsLH$ be the solutions to \eqref{eq:dg} and \eqref{eq:localMM}, respectively. Then 
  \begin{equation}
    \begin{aligned}
      \Enormh{u-u^{ms,L}_H} \leq &\Enormh{u-u_h} + \Ccontinuous\alpha^{-1/2}\|H(f-\Lp f)\|_{L^2(\Omega)} \\
      & + C_5\|H^{-1}\|_{L^\infty(\Omega)}L^{d/2}\gamma^L\|f\|_{L^2(\Omega)}
    \end{aligned}
  \end{equation}
  holds, where $L$ is the size of the patches, $C_1$ is a constant defined in \thmref{thm:ideal}, $0<\gamma<1$ and $C_5 = C_4^{1/2}C_2\Cbas\ba$, where $C_4=\Ccontinuous^2\Ccutoff^2(1+\Cbubble\Cstable)^2$ is defined in \lemref{lem:decayFunc}, and $C_2$ and $\gamma$ are defined in \lemref{lem:decay}.
\end{theorem}
\begin{proof}
  See Section \ref{sec:proofs}.
\end{proof}

\begin{remark}  \thmref{thm:convergence} is simplified to,
  \begin{equation}
    \Enormh{u-u^{ms,L}_H} \leq \Enormh{u-u_h} + C_1\|H\|_{L^\infty(\Omega)}.
  \end{equation}
  given that the patch size is chosen as $L = \lceil
  C\log(H^{-1})\rceil$ with an appropriate $C$ and $\|f\|_{L^2}=1$. In
  the numerical experiments we choose $C=2$.
\end{remark}

\begin{remark}
  If the convective term is small it is not necessary to include it in
  the computation of the correctors \cite{HMP14}. Instead the
  following correctors can be used: for all $\{T\in\T_H, j=1,\dots,r\}$ find
  $\hat\phi^L_{T,j}\in\Vfhw{\omega^L_T}$ such that
\begin{equation}\label{eq:locaBasis}
  a_h^{\mathrm{d}}(\hat\phi^L_{T,j},v) = a_h^{\mathrm{d}}(\lambda_{T,j},v),\quad\text{for all }v\in\Vfhw{\omega^L_T}.
\end{equation}
This gives the right convergence results if
\begin{equation}
  \mathcal{O}\left(\frac{\|\mathbf{b}\|_{L^\infty(\Omega)}}{\alpha}\right) = 1
\end{equation}
compared to \eqref{eq:size_convection} if the convective term is
included.
\end{remark}

\section{Numerical experiment}\label{sec:numerics}

We consider the domain $\Omega=[0,1]\times[0,1]$ and the forcing
function $f = 1+ \cos(2\pi x)\cos(2\pi y)$. The localization parameter
which determine the size of the patches is chosen as $L =\lceil
2\log(H^{-1})\rceil$, i.e., the size of the patches are
$2H\log(H^{-1})$.  Consider a coarse quadrilateral mesh, $\T_H$, of
size $H=2^{-i}$, $i=2,3,4,5$. The corrector functions are solved on
sub-grids of the quadrilateral mesh, $\T_h$, where $h=2^{-7}$.  We
consider three different permeabilities: $A_1 = 1 $, $A_2=A_2(y)$
which is piecewise constant with respect to a Cartesian grid of width
$2^{-6}$ in y-direction taking the values $1$ or $0.01$, and $A_3 =
A_3(x, y)$ which is piecewise constant with respect to a Cartesian
grid of width $2^{-6}$ both in the x- and y-directions, bounded below
by $\alpha = 0.05$ and has a maximum ratio $\beta/\alpha =
4\cdot10^{5}$. The permeability $A_3$ is taken from the $31$ layer in
the SPE 10 benchmark problem, see
\texttt{http:www.spe.org/web/csp/}. The diffusion coefficients $A_2$
and $A_3$ are illustrated in \fgref{fig:PeriodSPE}.
\begin{figure}
  \centering
  \subfigure[$A_2$]{
    \includegraphics[width=0.45\textwidth]{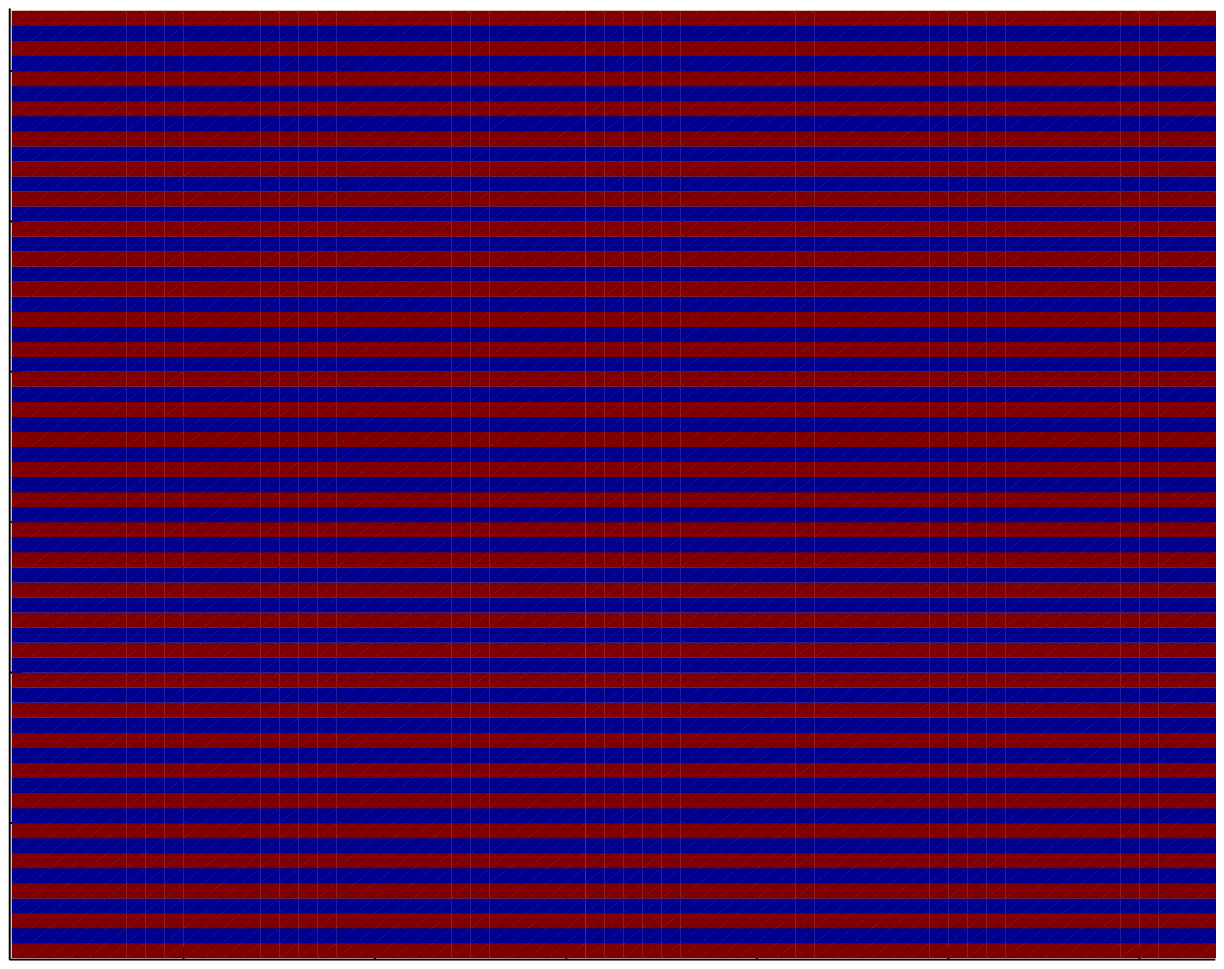}
  }
  \subfigure[$A_3$]{
    \includegraphics[width=0.45\textwidth]{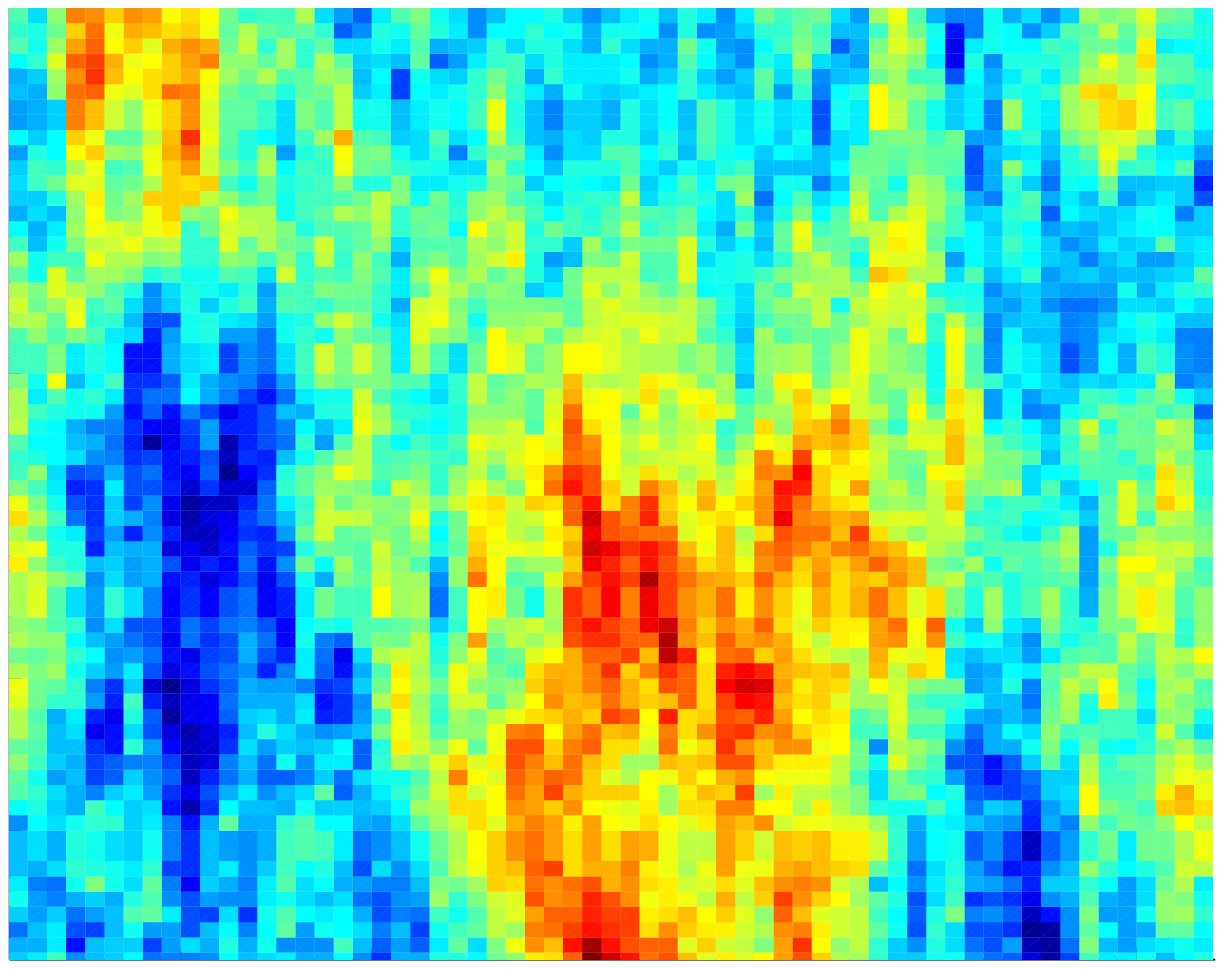}
  }
  \caption{The diffusion coefficients $A_2$ and $A_3$ in log scale.}
  \label{fig:PeriodSPE}
\end{figure}
For the convection term we consider: $\mathbf{b} = [C, 0]$, for different values of $C$. 

To investigate how the error in relative energy-norm, $\Enorm{u_h-u^{ms,L}_H}/\Enorm{u_h}$, depends on the magnitude of the convection we consider:  $A_1$ and $\mathbf{b} = [C, 0]$ with $C=\{32,64,128\}$. \fgref{fig:convection1} shows the convergence in energy-norm as a function of the coarse mesh size $H$ for the different values of $C$.
\begin{figure}[!htb]
\centering
\includegraphics[width=0.8\textwidth]{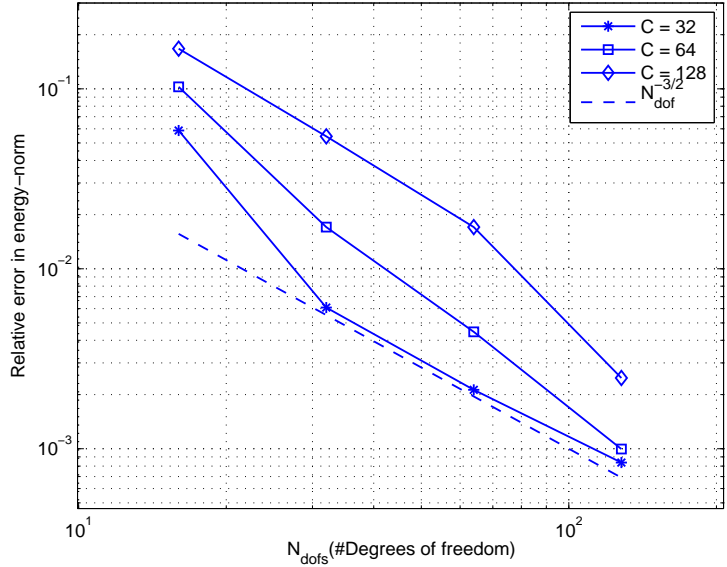}
\caption{The number degrees of freedom ($N_{dof}$) vs. the relative error in energy-norm, for different sizes of the convection term, $C$.}
\label{fig:convection1}
\end{figure}

Also, to see the effect of heterogeneous diffusion of the error in the relative energy-norm, $\Enorm{u_h-u^{ms,L}_H}/\Enorm{u_h}$,  we consider: \fgref{fig:period} which shows the error in relative energy-norm using $A_2$ and $\mathbf{b}=[1,0]$ and \fgref{fig:spe} which shows the error in relative energy-norm using $A_3$ and $\mathbf{b}=[512,0]$.
\begin{figure}[!htb]
\centering
\includegraphics[width=0.8\textwidth]{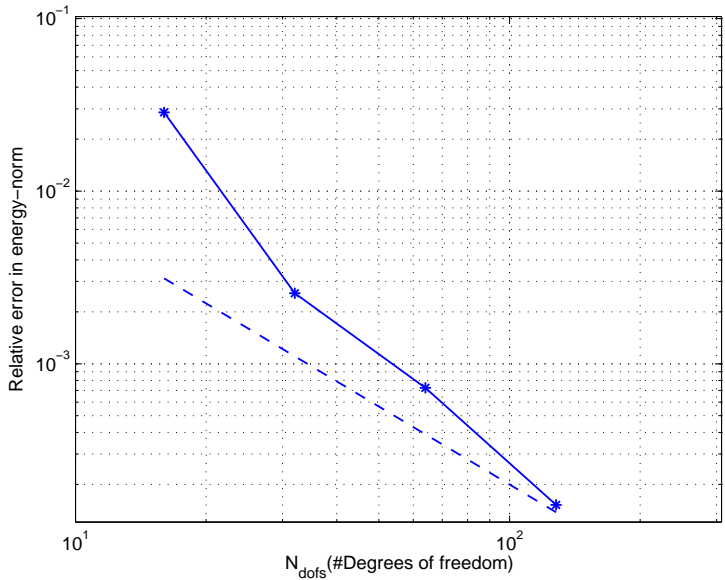}
\caption{The number degrees of freedom ($N_{dof}$) vs. the relative error in energy-norm, using a high contrast diffusion coefficients $A_2$ and $\mathbf{b}=[1,0]$. The dotted line corresponds to $N_{dof}^{-3/2}$.}
\label{fig:period}
\end{figure}
\begin{figure}[!htb]
\centering
\includegraphics[width=0.8\textwidth]{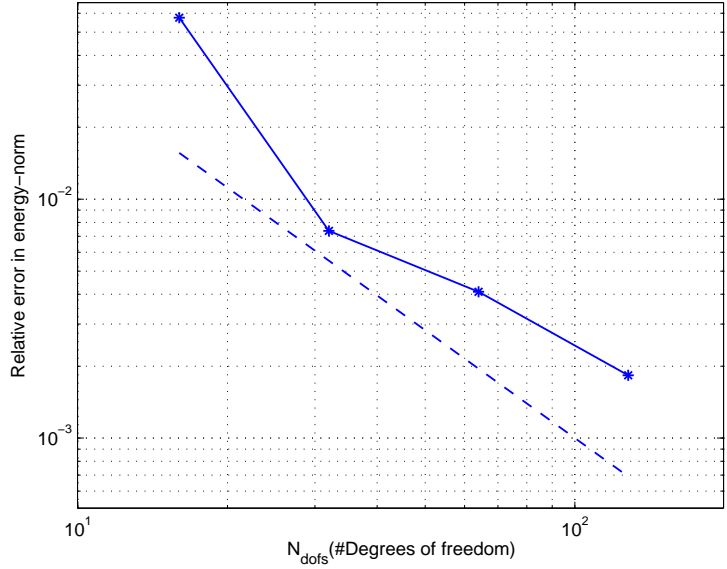}
\caption{The number degrees of freedom ($N_{dof}$) vs. the relative error in energy-norm, using a high contrast diffusion coefficients $A_3$ and $\mathbf{b}=[512,0]$. The dotted line corresponds to $N_{dof}^{-3/2}$.}
\label{fig:spe}
\end{figure}

We obtain $H^3$ convergence of the DG multiscale method to a reference solution in the relative energy-norm, $\Enorm{u_h-u^{ms,L}_H}/\Enorm{u_h}$, independent of the variation in the coefficients or regularity of the underlying solution.

\section{Proofs from Section~\ref{sec:MM}}\label{sec:proofs}
In this section we state the proofs of the main results which was
postponed from in section~\ref{sec:MM}. To this end we start by proving some
technical lemmas in Section~\ref{sec:tech_lem} which we use to prove
the main results in Section~\ref{sec:main_proof}.

\subsection{Technical lemmas}\label{sec:tech_lem}
In the proofs of the main results, \thmref{thm:ideal},
\lemref{lem:decay}, and \thmref{thm:convergence}, we will need some
definitions and technical lemmas stated below.

Continuity of the DG bilinear form for convective problems can be
proven on a orthogonal subset of $\Vh$. The space $\Vfh$ is an
orthogonal subset of $\Vh$ but on a coarse scale.
\begin{lemma}[Continuity in $(\Vh\times \Vfh)$ and $(\Vfh\times \Vh)$]\label{lem:continuous}
  For all, $(u,v)\in\Vfh\times\Vh$ or in $\Vh\times\Vfh$ , it holds
  \begin{equation}
    a(v,w) \lesssim \Ccontinuous\Enormh{v}\Enormh{w}
  \end{equation}
  where
  \begin{equation}
    \Ccontinuous = \ba+\|H\mathbf{b}\|_{L^\infty(\Omega)}\alpha^{-1}.
  \end{equation}
\end{lemma}
\begin{proof}
  Since $a^d_h$ is continuous in $(\Vh\times \Vh)$ with the constant
  $\ba$, continuity in $(\Vfh\times \Vh)$ follows from $\Vfh\subset
  \Vh$. For the convective part $a^{\operatorname*{c}}_h$, we have
  \begin{equation}
    \begin{aligned}
      a^{\operatorname*{c}}(v,w)& = \sum_{T\in \T_h}(\mathbf{b}\cdot\nabla v,w)_{L^2(T)}+\sum_{e\in\E_k}(b_e[v],[w])_{L^2(e)} \\ 
      &\qquad-\sum_{e\in\E_k(\Omega)}(\nu_e\cdot\mathbf{b}[v],\{w\})_{L^2(e)} +\sum_{e\in\E_k(\Gamma)}((\nu_e\cdot\mathbf{b})^\ominus v,w)_{L^2(e)} \\
      & \lesssim \sum_{T\in \T_h}\left( \|\mathbf{b}\|_{L^\infty(T)}\|\nabla v\|_{L^2(T)}\|w-\Lp w\|_{L^2(T)}\right)\hspace{-5pt}\\ 
      &\qquad+\sum_{e\in\E_k}\Big( \|\mathbf{b}\|_{L^\infty(e)}h^{-1/2}
\|[v]\|_{L^2(e)}\|w\|_{L^2(S^+\cup S^-)}\Big).
    \end{aligned}
  \end{equation}    
  where $S^+,S^-\in\T_h$ and $e=S^+\cap S^-$. Using a discrete
  Cauchy-Schwartz inequality and summing over the coarse elements, we
  get
  \begin{equation}
    \begin{aligned}
    a^{\operatorname*{c}}(v,w)  & \lesssim \alpha^{-1/2}\|H\mathbf{b}\|_{L^\infty(\Omega)}\Enormh{v}\|H^{-1}(w-\Lp w)\|_{L^2(\Omega)}, \\
      & \lesssim \|H\mathbf{b}\|_{L^\infty(\Omega)}\alpha^{-1}\Enormh{ v}\Enormh{ w},
    \end{aligned}
  \end{equation}
  which concludes the proof for $(\Vh\times\Vfh)$. The proof of $(\Vfh\times\Vh)$ is obtained by first integrating $(\mathbf{b}\cdot\nabla u,v)_{L^2(T)}$ by parts.
\end{proof}

The following cut-off function will be frequently used in the proof of
the main results.
\begin{definition}\label{def:cutoff}
  The function $\zeta^{d,D}\in P_o(\T_h)$, for $D>d$, is a cut off
  function fulfilling the following condition
  \begin{equation}
    \begin{aligned}
      \zeta^{d,D}_T\vert_{\omega^d_T} &= 1, \\
      \zeta^{d,D}_T\vert_{\Omega\setminus\omega^D_T} &= 0, \\
      \| [\zeta^{d,D}_T] \|_{L^\infty(\E_h(T))} &\lesssim \frac{\|h\|_{L^\infty(T)}}{(D-d)H\vert_T},
    \end{aligned}
  \end{equation}
  and $||[\zeta^{d,D}]||_{L^\infty(\partial(\omega^D_T\setminus\omega^d_T))} = 0$, for all $T\in\T_H$.
\end{definition}

For the cut off function has the following stability property.
\begin{lemma} \label{lem:cutoff}
  For any $v\in\V_h$ and $\zeta^{d,D}_T$ from \defref{def:cutoff}, the estimate,
  \begin{equation}
    \Enormh{\zeta^{d,D}_Tv} \lesssim \Ccutoff\Enormhw{v}{\omega^D_T},
  \end{equation}
holds, where $\Ccutoff=(\ba^2+\|h\mathbf{b}\|_{L^\infty(\Omega)}/\alpha)^{1/2}$.
\end{lemma}
\begin{proof}
  For the diffusion part we use the following result from \cite{EGMP13},
  \begin{equation}\label{eq:cuttoff1}
    \Enormd{(1-\zeta^{d,D}_T)v} \lesssim \ba \Enormhw{v}{\Omega\setminus\omega_T^{L-1}}
  \end{equation}
  and focus on the convective part. We obtain
  \begin{equation}\label{eq:cuttoff2}
    \begin{aligned}
      &\Enorma{(1-\zeta^{d,D}_T)v}^2 \\
      &= \sum_{e\in\E_h}\|b_e^{1/2}[(1-\zeta^{d,D}_T)v]\|^2_{L^2(e)} \\
      & \leq \sum_{\substack{e\in\E_h:\\ e\cap\omega^{L-1}_T\neq 0}}\left(\|b_e^{1/2}[v]\|^2_{L^2(e)}+\|h\|_{L^\infty(S^+\cup S^-)}^2 \|H^{-1}b_e^{1/2}\{v\}\|^2_{L^2(e)}\right) \\
      & \lesssim \sum_{\substack{T\in\T_H:\\ e\cap\omega^{L-1}_T\neq 0}}\hspace{-5pt}\|h\mathbf{b}\|_{L^\infty(T)}\left(\|h^{-1/2}[v]\|^2_{L^2(e)}+ \|H^{-1}(v-\Lp v)\|^2_{L^2(T)}\right) \\
      & \lesssim \frac{\|h\mathbf{b}\|_{L^\infty(\Omega)}}{\alpha} \Enormhw{v}{\Omega\setminus\omega^{L-1}}^2,
    \end{aligned}
  \end{equation}
  using $[vw]=\{v\}[w]+\{w\}[v]$, the triangle inequality, and a trace
  inequality. The proof is concluded using \eqref{eq:cuttoff1} and
  \eqref{eq:cuttoff2}.
\end{proof}

The following lemmas will be necessary in order to prove
\thmref{thm:convergence}.
\begin{lemma}\label{lem:decayFunc} The following estimate, 
  \begin{equation}
    \Enormh{\sum_{T\in\T_H,\, j=1,\dots,r} v_j(\phi_{T,j}-\phi^L_{T,j})}^2 \lesssim C_4L^{d}\sum_{T\in\T_H,\, j=1,\dots,r} |v_j|^2\Enormh{\phi_{T,j}-\phi^L_{T,j}}^2,
  \end{equation}
holds, where $C_4=\Ccontinuous^2\Ccutoff^2(1+\Cbubble\Cstable)^2$.
\end{lemma}
\begin{proof}
  The proof is analogous with the proof of Lemma 12 in \cite{EGMP13}.
\end{proof}

\subsection{Proof of main results}\label{sec:main_proof}

We are now ready to prove, \thmref{thm:ideal}, \lemref{lem:decay}, and
\thmref{thm:convergence}.
\begin{proof}[\thmref{thm:ideal}]
  Let us decompose $u_h$ into a coarse contribution, $v^{ms}_H\in\VmsH$, and a fine scale remainder, $v^f\in\Vfh$, i.e., $u_h = v^{ms}_H + v^f$. For $v^f$ we have
  \begin{equation}\label{eq:ideal1}
    \begin{aligned}
      \Enormh{v^f}^2 &\lesssim a_h(v^f,v^f)= a_h(u_h,v^f) = (f,v^f)_{L^2(\Omega)}\\
      & = (f-\Lp f,v^f-\Lp v^f)_{L^2(\Omega)}  \\
      & \leq \| H(f-\Lp f)\|_{L^2(\Omega)}\|H^{-1}(v^f-\Lp v^f)||_{L^2(\Omega)} \\
      & \leq \alpha^{-1/2}\| H(f-\Lp f)\|_{L^2(\Omega)}\Enormh{v^f}.
    \end{aligned}
  \end{equation}
  Using continuity, we get
  \begin{equation}
    \begin{aligned}
    \Enormh{u_h - u^{ms}_H}^2 &\lesssim a_h(u_h - u^{ms}_H,u_h - u^{ms}_H)=a_h(u_h - u^{ms}_H,u_h - v^{ms}_H) \\
    & \lesssim \Ccontinuous \Enormh{u_h - u^{ms}_H}\Enormh{u_h - v^{ms}_H},
    \end{aligned}
  \end{equation}
  which concludes the proof together with \eqref{eq:ideal1}.
\end{proof}

\begin{proof}[\lemref{lem:decay}]
  Define $e:=\phi_{T,j}-\phi_{T,j}^{L}$ where
  $\phi_{T,j}\in\Vfh$ and $\phi_{T,j}^L\in\Vfhw{\omega^L_T}$. We have
  \begin{equation}
      \enorm{e}^2_h \lesssim  a_h(e,\phi_{T,j}-\phi_{T,j}^{L}) = a_h(e,\phi_{T,j}-v)  \lesssim \Ccontinuous\Enormh{e}\Enormh{\phi_{T,j}-v}.
  \end{equation}
  Furthermore from \lemref{lem:bubble}, there exist a $v = \zeta^{L-1,L}_T \phi_{T,j} - b_T\in\Vfhw{\omega^{L}_T}$ such that $\Lp b_T = \Lp (\zeta^{L-1,L}_T \phi_{T,j})$ and $\Enormh{b_T}\lesssim\Cbubble\EnormH{\Lp(\zeta^{L-1,L}_T \phi_{T,j})}$, we have
  \begin{equation}
    \begin{aligned}\label{eq:decaySteg1_1}
      \enorm{e}_h \lesssim \Ccontinuous\left(\Enormh{(1-\zeta^{L-1,L}_T) \phi_{T,j}}+\Enormh{b_T}\right),
    \end{aligned}
  \end{equation}
  where
  \begin{equation}\label{eq:decaySteg1_2}
    \begin{aligned}
      \Enormh{b_T} & \lesssim \Cbubble \EnormH{\Lp\zeta^{L-1,L}_T \phi_{T,j} } = \Cbubble\EnormH{\Lp(1-\zeta^{L-1,L}_T)\phi_{T,j} } \\
      & \lesssim \Cbubble\Cstable \Enormh{(1-\zeta^{L-1,L}_T)\phi_{T,j} } \lesssim \Cbubble\Cstable\Ccutoff \Enormhw{\phi_{T,j}}{\Omega\setminus\omega^{L-1}_T}.
    \end{aligned}
  \end{equation}
  using \lemref{lem:bubble}, \lemref{lem:stable}, and \lemref{lem:cutoff}.
  We obtain,
  \begin{equation}\label{eq:decaySteg1}
    \enorm{e}_h \lesssim C_2 \Enormhw{\phi_{T,j}}{\Omega\setminus\omega^{L-1}_T},
  \end{equation}
  where $C_2 = \Ccontinuous\Ccutoff(1+\Cbubble\Cstable)$ from \eqref{eq:decaySteg1_1} and \eqref{eq:decaySteg1_2}.
  
  The next step in the proof is to construct a recursive relation
  which will be used to prove the decay of the correctors.  To this
  end, let $\ell k = L-1$, and define another the cut off function,
  $\eta_T^m:=(1-\zeta^{\ell (k-m-1)-m,\ell (k - m)-m})$ and the patch
  $\widetilde\omega^m_T:=\omega_T^{\ell(k-m+1)-m}$, for $m =
  0,1,\dots, \lfloor\ell k/(\ell+1)-1\rfloor$. Note that
  $\widetilde\omega^{m+1}_T\subset\widetilde\omega^m_T$. We obtain
  \begin{equation}
    \begin{aligned}
      \Enormhw{\phi_{T,j}}{\Omega\setminus\widetilde\omega^m_T} \leq\Enormh{\eta_T^m\phi_{T,j}}\lesssim a_h(\eta^m_T\phi_{T,j},\eta^m_T\phi_{T,j}).
    \end{aligned}
  \end{equation}
  To shorten the proof we refer to the following inequality
  \begin{equation}\label{eq:decayD}
    a^{\operatorname*{d}}(\eta^m_T\phi_{T,j},\eta^m_T\phi_{T,j}) \lesssim a^{\operatorname*{d}}(\phi_{T,j},(\eta^m_T)^2\phi_{T,j}-b_T) + \frac{\ba^2}{\ell}\Enormhw{\phi_{T,j}}{\widetilde\omega^{m}_T\setminus\widetilde\omega^{m+1}_T}^2.
  \end{equation}
  where $(\eta^m_T)^2\phi_{T,j}- b_T\in \Vfh$, in the proof of Lemma 10 in \cite{EGMP13}.
  We focus on the convection term,
  since the cut of function is piecewise constant it follows that
  \begin{equation}\label{eq:constancutoff}
    \begin{aligned}
      (\mathbf{b}\cdot\nabla \eta^m_T\phi_{T,j},\eta^m_T\phi_{T,j})_{L^2(S)}  =  (\mathbf{b}\cdot\nabla \phi_{T,j},(\eta^m_T)^2\phi_{T,j})_{L^2(S)}
    \end{aligned}
  \end{equation}
  for all $S\in\T_h$. Using the following equalities from (Appendix A in \cite{EGMP13})
  \begin{equation}
    \begin{aligned}
      \{vw\}[vw] &= \{w\}[v^2w]-[v]\{w\}\{v\}\{w\}+1/4[v]\{v\}[w][w],\\
      [vw][vw] & = [w][v^2w] - 1/4[v]^2[w]^2 + [v]^2\{w\}^2,
    \end{aligned}
  \end{equation} 
  and  \eqref{eq:constancutoff}, we obtain
  \begin{equation}\label{eq:stepontheway}
    \begin{aligned}
      &a^{\operatorname*{c}}(\eta^m_T\phi_{T,j},\eta^m_T\phi_{T,j}) =  a^{\operatorname*{c}}(\phi_{T,j},(\eta^m_T)^2\phi_{T,j}) \\
      & + \sum_{\substack{e\in\E_h(\Omega)}}\Big( (\nu_e\cdot\mathbf{b}[\eta^m_T]\{\phi_{T,j}\},\{\eta^m_T\}\{\phi_{T,j}\})_{L^2(e)} \\
      &\qquad \qquad-1/4(\nu_e\cdot\mathbf{b}[\eta^m_T]\{\phi_{T,j}\},\{\eta^m_T\}[\phi_{T,j}])_{L^2(e)} \\
      &\qquad \qquad -1/4(b_e[\eta^m_T]^2,[\phi_{T,j}]^2)_{L^2(e)}+(b_e[\eta^m_T]^2,\{\phi_{T,j}\}^2)_{L^2(e)}\Big).
    \end{aligned}
  \end{equation}
  The sum over the edges terms can be bounded using that
  $\|[\eta^m_T]\|_{L^\infty(T)}\lesssim \|h\|_{L^\infty(T)}/H\vert_T$,
  $\|\{\eta^m_T\}\|_{L^\infty(\Omega)}\lesssim 1$,
  $\|h\|_{L^\infty(T)}/H\vert_T\ell<1$, and a trace inequality. We obtain
  \begin{equation}\label{eq:decay2}
    \begin{aligned}
      &\sum_{\substack{e\in\E_h(\Omega):\\e\cap(\widetilde\omega^{m}_T\setminus\widetilde\omega_T^m)\neq 0}}\frac{\|H^{-1}\mathbf{b}\|_{L^\infty(e)}}{\ell}\bigg(\|h^{1/2}\{\phi_{T,j}\}\|_{L^2(e)}\|h^{1/2}\{\phi_{T,j}\}\|_{L^2(e)}+\\
      &\qquad\|h^{1/2}\{\phi_{T,j}\}\|_{L^2(e)}\|h^{1/2}[\phi_{T,j}]\|_{L^2(e)}  + \|h^{1/2}[\phi_{T,j}]\|^2_{L^2(e)}  \\ &\qquad+\|h^{1/2}\{\phi_{T,j}\}\|^2_{L^2(e)}\bigg)\\
      &\lesssim\sum_{\substack{e\in\E_H(\Omega):\\e\cap(\widetilde\omega^{m}_T\setminus\widetilde\omega_T^m)\neq 0}}\frac{\|H^{-1}\mathbf{b}\|_{L^\infty(e)}}{\ell}\|\phi_{T,j}\|_{L^2(T^+\cup T^-)}^2 \\
      &\lesssim\sum_{\substack{T\in\T_H:\\T\cap(\widetilde\omega^{m}_T\setminus\widetilde\omega_T^m)\neq 0}}\frac{\|H\mathbf{b}\|_{L^\infty(T)}}{\ell}\|H^{-1}(\phi_{T,j}-\Lp\phi_{T,j})\|_{L^2(T)}^2 \\
      &\lesssim \frac{\|H\mathbf{b}\|_{L^\infty(\Omega)}}{\ell\alpha}\Enormhw{\phi_{T,j}}{(\widetilde\omega^{m}_T\setminus\widetilde\omega_T^{m+1})}^2.
    \end{aligned}
  \end{equation}
  Combining the results, we have
  \begin{equation}
    \begin{aligned}
      &\Enormhw{\phi_{T,j}}{\Omega\setminus\widetilde\omega^m_T}^2 \lesssim a(\phi_{T,j},(\eta^m_T)^2\phi_{T,j}-b_T) +  a(\phi_{T,j},b_T) \\
      & \qquad+ \ell^{-1}\left(\ba^2 +\frac{\|H\mathbf{b}\|_{L^\infty(\Omega)}}{\alpha} \right)\Enormhw{\phi_{T,j}}{(\widetilde\omega^{m}_T\setminus\widetilde\omega_T^{m+1})}^2 ,
    \end{aligned}
  \end{equation}
  where $b_T$ has support in
  $\widetilde\omega^m_T\setminus\widetilde\omega^{m+1}_T$, such that
  $(\eta^m_T)^2\phi_{T,j}-b_T\in\Vfh$ and
  $\Enormh{b_T}\lesssim\Cbubble \EnormH{\Lp
    ((\eta^m_T)^2\phi_{T,j})}$, see \lemref{lem:bubble}. We have
  \begin{equation}
    a(\phi_{T,j},(\eta^m_T)^2\phi_{T,j}-b_T) = 0.
  \end{equation}
  For all $T\in\T_H$ the operator $\Lp$ is stable in the
  $L^2(T)$-norm, we have
  \begin{equation}\label{eq:decay4}
    \begin{aligned}
      &\Enormhw{b_T}{\widetilde\omega^{m}_T\setminus\widetilde\omega^{m+1}_T}^2 \lesssim \Cstable^2\EnormH{\Lp((\eta^m_T)^2\phi_{T,j})}^2 \\
      &=\Cstable^2\left(\EnormdH{\Lp((\eta^m_T)^2\phi_{T,j})}^2+\EnormaH{\Lp((\eta^m_T)^2\phi_{T,j})}^2\right).
    \end{aligned}
  \end{equation}
  For the first term in \eqref{eq:decay4} we refer to the result
  \begin{equation}
    \EnormdH{\Lp((\eta^m_T)^2\phi_{T,j})}^2 \lesssim \frac{\ba^2}{\ell^2}\Enormhw{\phi_{T,j}}{\widetilde\omega^{m}_T\setminus\widetilde\omega^{m+1}_T}^2,
  \end{equation}
  from \cite{EGMP13} and for the second them we have
  \begin{equation}
    \begin{aligned}
      &\EnormaH{\Lp((\eta^m_T)^2\phi_{T,j})}^2=\EnormaH{\Lp((\eta^m_T-\Pi_0\eta^m_T)^2\phi_{T,j})}^2 \\
      & = \sum_{\substack{e\in\E_H(\Omega)}} \|b_e^{1/2}[\Lp((\eta^m_T)^2-\Pi_0(\eta^m_T)^2)\phi_{T,j})]\|^2_{L^2(e)} \\     
      & = \sum_{T\in\T_H} \|H^{-1}\mathbf{b}\|_{L^\infty(T)}\|(\eta^m_T)^2-\Pi_0(\eta^m_T)^2\|^2_{L^\infty(T)}\|\phi_{T,j}-\Lp\phi_{T,j}\|^2_{L^2(T)} \\
      & \lesssim\frac{\|H\mathbf{b}\|_{L^\infty(T)}}{\alpha\ell^{2}}\Enormhw{\phi_{T,j}}{\widetilde\omega^{m}_T\setminus\widetilde\omega^{m+1}_T}^2.
    \end{aligned}
  \end{equation}
  We obtain 
  \begin{equation}\label{eq:decay5}
    \begin{aligned}
      &\Enormhw{\phi_{T,j}}{\Omega\setminus\omega^{m}_T}^2 \lesssim \ell^{-1}\left(\ba^2+\frac{\|H\mathbf{b}\|_{L^\infty(T)}}{\alpha}\right)\Enormhw{\phi_{T,j}}{\Omega\setminus\omega_T^{m+1}}^2 \\
      & =  C_3\ell^{-1}\Enormhw{\phi_{T,j}}{\Omega\setminus\widetilde\omega_T^{m+1}}^2
    \end{aligned}
  \end{equation}
  where $C_3 = C(\ba^2+\|H\mathbf{b}\|_{L^\infty(\Omega)}\alpha^{-1})$
  and $C$ is the generic constant hidden in '$\lesssim$'.  We have
  \begin{equation}\label{eq:decayR}
    \Enormhw{\phi_{T,j}}{\Omega\setminus\widetilde\omega^{m}_T}^2 \lesssim C_3\ell^{-1}\Enormhw{\phi_{T,j}}{\Omega\setminus\widetilde\omega^{m+1}_T}^2,
  \end{equation}
  for any $m=0,1,\dots,\lfloor\ell k/(\ell+1)\rfloor-1$, which we can
  use recursively as
  \begin{equation}
    \begin{aligned}
    \Enormhw{\phi_{T,j}}{\Omega\setminus\widetilde\omega^{1}_T}^2&\lesssim (C_3\ell^{-1})^{k-1}\Enormhw{\phi_{T,j}}{\Omega\setminus\widetilde\omega^{k}_T}^2  \\ &= (C_3\ell^{-1})^{\lfloor\ell k/(\ell+1)\rfloor-1}\Enormhw{\phi_{T,j}-\lambda_{T,j}}{\Omega}^2.
    \end{aligned}
  \end{equation}
  Note that $k/2$ is a lower bound of $\ell k/(\ell+1)$. Equation \eqref{eq:decaySteg1}
  together with \eqref{eq:decayR}, gives
  \begin{equation}
    \Enormh{\phi_{T,j}-\phi^L_h} \lesssim C_2(C_3\ell^{-1})^{\frac{1}{2}(\ell k/(\ell +1)-1)}\Enormh{\phi_{T,j}-\lambda_{T,j}}.
  \end{equation}
  which concludes the proof is concluded.
\end{proof}

\begin{proof}[\thmref{thm:convergence}]
 Using the triangle inequality, we have
  \begin{equation}
    \Enormh{u-u^{ms,L}_H} \leq \Enormh{u-u_h} + \Enormh{u_h-u^{ms,L}_H}.
  \end{equation}
  Note that, $u_h\in\Vh$, can be decomposed into a coarse, $v^{ms}_H\in\VmsH$, and a fine, $v^f\in\Vfh$, scale contribution, i.e.,  $u_h=v^{ms}_H + u^f$. Also, let $v^{ms,L}_H\in\VmsLH$ be chosen such that $\Lp v^{ms,L}_H = \Lp v^{ms}_H$.
 We have
  \begin{equation}
    \begin{aligned}
      \Enormh{u_h-u^{ms,L}_H} &\lesssim  a_h(u_h-u^{ms,L}_H,u_h-u^{ms,L}_H) \\
      &=  a_h(u_h-u^{ms,L}_H,u_h-v^{ms,L}_H) \\
      &\lesssim \Ccontinuous \Enormh{u_h-u^{ms,L}_H}\Enormh{u_h-v^{ms,L}_H},
    \end{aligned}
  \end{equation}
  and obtain
  \begin{equation}\label{eq:conv1}
    \begin{aligned}
      \Enormh{u-u^{ms,L}_H}\leq &\Enormh{u-u_h} \\
      &   + \Ccontinuous\left(\Enormh{u_h-v^{ms}_H} + \Enormh{v^{ms}_H-v^{ms,L}_H}\right).
    \end{aligned}
  \end{equation}
  The first term in \eqref{eq:conv1} implies that the reference mesh need to be sufficiently fine to get a sufficient approximation. The second term is approximated using \eqref{eq:ideal1}, i.e.
  \begin{equation}
    \Enormh{u_h-v^{ms}_H} \lesssim \alpha^{-1/2}\|H(1-\Lp)f\|_{L^2(\Omega)},
  \end{equation}
and for the last term in we have,
\begin{equation}
  \begin{aligned}
    \Enormh{v^{ms}_H-v^{ms,L}_H}^2 &= \Enormh{\sum_{T\in\T_H,\, j=1,\dots,r}v^{ms}_{H,T}(x_j)(\phi_{T,h}-\phi^L_{T,j})}^2 \\
    &\lesssim C_4L^d\sum_{T\in\T_H,\, j=1,\dots,r} |v^{ms}_{H,T}(x_j)|^2\Enormh{\phi_{T,h}-\phi^L_{T,j}}^2 \\
    &\lesssim C_4C_2^2L^d\gamma^{2L}\sum_{T\in\T_H,\, j=1,\dots,r} |v^{ms}_{H,T}(x_j)|^2\Enormh{\phi_{T,h}-\lambda_{T,j}}^2,
  \end{aligned}
\end{equation}
using \lemref{lem:decayFunc} and \lemref{lem:decay}.

We obtain, using \lemref{lem:bas}, that
\begin{equation}
  \begin{aligned}
    & \sum_{T\in\T_H,\, j=1,\dots,r} |v^{ms}_{H,T}(x_j)|^2\Enormh{\phi_{T,h}-\lambda_{T,j}}^2 \\
    & \leq  \Cbas^2  \sum_{T\in\T_H,\, j=1,\dots,r} \|H^{-1}v^{ms}_{H,T}(x_j)\lambda_{T,j}\|^2_{L^2(\Omega)} \\
    & \lesssim  \Cbas^2\beta \|\sum_{T\in\T_H,\, j=1,\dots,r}H^{-1}v^{ms}_{H,T}(x_j)\lambda_{T,j}\|^2_{L^2(\Omega)} \\
    & = \Cbas^2\beta  \|\sum_{T\in\T_H,\, j=1,\dots,r}H^{-1}v^{ms}_{H,T}(x_j)\Lp(\lambda_{T,j}-\phi_{T,j})\|^2_{L^2(\Omega)} \\
    &\leq  \Cbas^2\beta\|H^{-1}\|_{L^\infty(\Omega)} \|\Lp (v^{ms}_H+u^f)\|^2_{L^2(\Omega)} \\
    &\leq  \Cbas^2\beta \|H^{-1}u_h\|_{L^2(\Omega)}^2 \\
    &\leq \Cbas^2\ba^2\|H^{-1}\|_{L^\infty(\Omega)}\Enormh{u_h}. \\
  \end{aligned}
\end{equation}
holds and we conclude the proof.
 
\end{proof}

\bibliographystyle{plain}
\bibliography{references}
\end{document}